\documentclass[reqno]{amsart}
\usepackage{amsmath,amsthm,latexsym,amssymb,hyperref, amsfonts}
\usepackage{stmaryrd}
\usepackage{eucal}
\usepackage{mathrsfs}
\usepackage[all]{xy}
\usepackage{color}
\usepackage{tikz}
\usepackage{url}
\usepackage{tensor}
\usepackage{comment}
\usepackage{tikz-cd}
\usepackage{enumerate}
\usepackage{enumitem}
\usepackage{mathrsfs}
\usepackage{dutchcal}
\usetikzlibrary{matrix,arrows,decorations.pathmorphing}
\title[The Coalgebraic Enrichment of Algebras in Higher Categories]{The Coalgebraic Enrichment of Algebras in Higher Categories}

\author[Maximilien P\'eroux]{Maximilien P\'eroux}

\theoremstyle{definition}
\newtheorem{defi}{Definition}[section]

\newtheorem{ex}[defi]{Example}
\newtheorem{rem}[defi]{Remark}

\numberwithin{equation}{section}

\theoremstyle{plain}
\newtheorem{thm}[defi]{Theorem}
\newtheorem{prop}[defi]{Proposition}
\newtheorem{lem}[defi]{Lemma}
\newtheorem{cor}[defi]{Corollary}

\renewcommand{\o}{\otimes}
\newcommand{\bI}{\mathbb I}
\newcommand{\I}{\mathbb I}

\newcommand{\ei}{{\mathbb{E}_\infty}}

\renewcommand{\r}{\rightarrow}

\newcommand{\C}{\mathsf{C}}

\newcommand{\op}{^\mathsf{op}}
\newcommand{\M}{\mathsf{M}}

\newcommand{\N}{\mathscr{N}}

\renewcommand{\lim}{\mathsf{lim}}
\newcommand{\colim}{\mathsf{colim}}

\newcommand{\prl}{\mathcal{Pr}^L}

\newcommand{\coalg}{\mathsf{CoAlg}}
\newcommand{\ccoalg}{\mathsf{CoCAlg}}

\newcommand{\sweedler}{\triangleright}

\newcommand{\Dinf}{\mathcal{D}}

\newcommand{\Cinf}{\mathcal{C}}
\newcommand{\Oinf}{\mathcal{O}}
\newcommand{\coalginf}{\mathcal{CoAlg}}
\newcommand{\alginf}{\mathcal{Alg}}

\tikzset{
    labl/.style={anchor=south, rotate=-32, inner sep=.9mm}
}

\tikzset{
    labll/.style={anchor=south, rotate=32, inner sep=.9mm}
}

\begin{document}

\address{Department of Mathematics, University of Pennsylvania,
209 South 33rd Street,
Philadelphia, PA, 19104-6395, USA}
    \email{mperoux@sas.upenn.edu}

   

\begin{abstract} 
We prove that given  $\Cinf$ a presentably symmetric monoidal $\infty$-cate\-gory, and any essentially small $\infty$-operad $\Oinf$, the $\infty$-category of $\Oinf$-algebras in $\Cinf$ is enriched, tensored and cotensored over the presentably symmetric monoidal $\infty$-category of $\Oinf$-coalgebras in $\Cinf$. We provide a higher categorical analogue of the universal measuring coalgebra.
For categories in the usual sense, the result was proved by Hyland, L\'{o}pez Franco, and Vasilakopoulou.
\end{abstract}

\maketitle

\section{Introduction}

The dual of a coalgebra is always an algebra. However, unless we require the algebra to be finite dimensional, the dual of an algebra is not a coalgebra.
The universal measuring coalgebra was introduced in \cite{sweedler} as a way to balance this issue.
In ordinary categories, the measuring provides an enrichment for algebras over coalgebras: this was established in \cite[5.2]{hopfmeasure} and \cite[2.18]{cocat}. We provide here, in Theorem \ref{thm: alg enriched over coalg}, its $\infty$-categorical analogue. In any presentably symmetric monoidal $\infty$-category, the algebra objects are enriched, tensored and cotensored over coalgebras. 
Therefore spaces of algebra morphisms are endowed with a rich structure. We use the notion of enriched $\infty$-categories following \cite{gepner-haugseng} and \cite{hinichenr}.

Algebras in $\infty$-categories formalize the notion of homotopy coherent associative and unital algebras, see \cite{lurie1}. 
Following \cite{lurie2}, we provide a general dual definition of coalgebras in $\infty$-categories. These are objects with a comultiplication that is coassociative up to higher homotopies. We show, in Proposition \ref{prop: presentability of inf coalg}, that if an $\infty$-operad $\Oinf$ is essentially small, the $\infty$-category of $\Oinf$-coalgebras in a presentable $\infty$-category remains presentable.

A similar result would be very challenging to prove in model categories. 
Let $\M$ be a combinatorial symmetric monoidal model category. 
Suppose we have a model structure for algebras $\mathsf{Alg}(\M)$ in $\M$ and a model structure for coalgebras $\mathsf{CoAlg}(\M)$ in $\M$, in which the weak equivalences in both of these models are created by their underlying functor.
One analogous result would be to show that $\mathsf{Alg}(\M)$ is a $\mathsf{CoAlg}(\M)$-model category, in the sense of \cite[4.2.18]{hovey}.
There are several issues with that. A left-induced model structure on $\coalg(\M)$ may not always exist, and when it does, $\M$ may have been replaced by a Quillen equivalent model category that is not a monoidal model category, see \cite{left2}. 
Even in cases where we can left-induce from a monoidal model category, the homotopy theory associated to $\coalg(\M)$ may not be the correct one, see  \cite{perouxshipley} and \cite{coalginDK}. 

\subsection*{Acknowledgement}
The results here are part of my PhD thesis \cite{phd}, and as such, I would like to express my gratitude to my advisor Brooke Shipley for her help and guidance throughout the years. I would also like to thank Rune Haugseng for clarifying and answering many of my questions. I am also thankful for many fruitful conversations with Shaul Barkan and tslil clingman that sparked results in this paper. I thank the referee for helpful comments on the preliminary version of this paper.

\section{Presentability of Coalgebras}
We present here the formal definition of coalgebras in $\infty$-categories, generalizing \cite[Section 3.1]{lurie2}, which was for the case of $\mathbb{E}_\infty$-coalgebras. We define and extend the results for coalgebras over any $\infty$-operad.
Our main result in this section is that coalgebras of a presentably symmetric monoidal $\infty$-category also form a presentable $\infty$-category, see Corollary \ref{cor: presentably imply presentable}.

We invite the reader to review the definition of a symmetric monoidal $\infty$-category in \cite[2.0.0.7]{lurie1}. More generally, for any $\infty$-operad $\Oinf$ (see \cite[2.1.1.10]{lurie1}), we will consider the notion of an \emph{$\Oinf$-monoidal $\infty$-category} as in \cite[2.1.2.15]{lurie1}. If we choose $\Oinf$ to be the commutative $\infty$-operad (\cite[2.1.1.18]{lurie1}), then $\Oinf$-monoidal $\infty$-categories are precisely symmetric monoidal $\infty$-categories.

\begin{defi}\label{defi: coalgebras in inf cat}
Let $\Oinf$ be an $\infty$-operad.
Let $\Cinf$ be an $\mathcal{O}$-monoidal $\infty$-category. An \emph{$\mathcal{O}$-coalgebra object in $\Cinf$} is an $\mathcal{O}$-algebra object in $\Cinf\op$. The $\infty$-category of $\mathcal{O}$-coalgebra objects in $\Cinf$ is defined as the $\infty$-category $\coalginf_\Oinf(\Cinf):={\left(\alginf_\Oinf(\Cinf\op) \right)}\op.$
More generally, given any map $\Oinf'^\otimes\rightarrow \Oinf^\otimes$ of $\infty$-operads , we define the $\infty$-category of $\Oinf'$-coalgebras in $\Cinf$ as $\coalginf_{\Oinf'/\Oinf}(\Cinf)=(\alginf_{\Oinf'/\Oinf}(\Cinf\op))\op$.
\end{defi}

\begin{rem}\label{rem: True definition of coalg in inf}
If $\Cinf$ is an $\mathcal{O}$-monoidal $\infty$-category, then $\Cinf\op$ can be given  an $\mathcal{O}$-monoidal structure uniquely up to contractible choice, as in \cite[2.4.2.7]{lurie1}. 
One can use the work of \cite{dualcocart} to give an explicit choice of the coCartesian fibration for $\Cinf\op$. 
For instance, let $p:\Cinf^\o\rightarrow \Oinf^\o$ be the coCartesian fibration associated to the symmetric monoidal structure of $\Cinf$.
Then straightening of the coCartesian fibration gives a functor: \[F:\Oinf^\o \longrightarrow \widehat{\mathcal{Cat}}_\infty,\]
where $\widehat{\mathcal{Cat}}_\infty$ is the $\infty$-category of (not necessarily small) $\infty$-categories, as in \cite[3.0.0.5]{lurie1}. Then, by \cite[1.5]{dualcocart}, the functor $F$ also classifies a Cartesian fibration:  \[p^\vee :(\Cinf^\o)^\vee \longrightarrow (\Oinf^\o)\op.\] 
An explicit construction is given in \cite[1.7]{dualcocart}. The opposite map:
\[(p^\vee)\op:( (\Cinf^\o)^\vee)\op \longrightarrow \Oinf^\o,\]
is a coCartesian fibration that is classified by:
\[
\begin{tikzcd}
\Oinf^\o \ar{r}{F} & \widehat{\mathcal{Cat}}_\infty \ar{r}{\mathsf{op}} & \widehat{\mathcal{Cat}}_\infty.
\end{tikzcd}
\]
One can check that the fiber of $(p^\vee)\op$ over $X$ in $\Oinf$ is equivalent to $(\Cinf_X)\op$, and thus gives $\Cinf\op$ a $\Oinf$-monoidal structure.
We see that $\Oinf$-coalgebras are sections of the Cartesian fibration $p^\vee:(\Cinf^\o)^\vee\rightarrow (\Oinf^\o)\op$ that  sends inert morphisms in $(\Oinf^\o)\op$ to $p^\vee$-Cartesian morphisms in $(\Cinf^\o)^\vee$. 
\end{rem}

\begin{rem}
Recall from \cite[2.0.0.1]{lurie1} that given any symmetric monoidal (ordinary) category $\C$, one can define a category $\C^\otimes$, such that the nerve $\N(\C^\o)$ is a symmetric monoidal $\infty$-category whose underlying $\infty$-category is $\N(\C)$, see \cite[2.1.2.21]{lurie1}.
If we denote by $\coalg(\C)$ the category of coassociative and counital coalgebras in $\C$, then, dually from \cite[4.21]{groth}, we obtain:
\[
\coalginf_{\mathbb{A}_\infty}\left(\N(\C)\right) \simeq \N\left(\coalg(\C)\right).
\]
Similarly, if we denote by $\ccoalg(\C)$ the category of cocommutative coalgebras in $\C$, we obtain:
\[
\coalginf_{\mathbb{E}_\infty}\left(\N(\C)\right) \simeq \N\left(\ccoalg(\C)\right).
\]
\end{rem}

\begin{prop}[{\cite[3.2.4.4]{lurie1}}]\label{prop: coalgebra are sym mon}
Let $\Oinf$ be an $\infty$-operad.
Let $\Cinf$ be an $\Oinf$-monoidal $\infty$-category.
Then the $\infty$-category $\alginf_\Oinf(\Cinf)$ inherits a $\Oinf$-monoidal structure, given by pointwise tensor product. Dually, the $\infty$-category $\coalginf_\Oinf(\Cinf)$ inherits a $\Oinf$-monoidal structure, given by pointwise tensor product.
\end{prop}

\begin{prop}\label{prop: colimits of coalgebra}
Let $\Cinf$ be a $\Oinf$-monoidal $\infty$-category and let $K$ be a simplicial set. If, for each $X$ in $\Oinf$, the fiber $\Cinf_X$ admits $K$-indexed colimits, then the $\infty$-category $\coalginf_\Oinf(\Cinf)$ admits $K$-indexed colimits, and the forgetful functor $U:\coalginf_\Oinf(\Cinf)\rightarrow \Cinf$ preserves $K$-indexed colimits.
\end{prop}

\begin{proof}
Apply \cite[3.2.2.5]{lurie1} to the coCartesian fibration $(p^\vee)\op:( (\Cinf^\o)^\vee)\op \longrightarrow \Oinf^\o$ defined in Remark \ref{rem: True definition of coalg in inf}.
\end{proof}

Recall the definition \cite[5.5.0.1]{htt} of a presentable $\infty$-category. Denote $\prl$ the $\infty$-category of presentable $\infty$-categories with small colimit preserving functors. It is endowed with a symmetric monoidal structure (\cite[4.8.1.15]{lurie1}). 

\begin{defi}\label{defi: presentably O-monoidal}
An $\infty$-category $\Cinf$ is said to be \emph{presentably $\Oinf$-monoidal} if it is an $\Oinf$-algebra in $\prl$, i.e., $\Cinf$ is $\Oinf$-monoidal, for each object $X$ in $\Oinf^\otimes$, the fiber $\Cinf^\otimes_X$ is presentable, and for every morphism $f:X\rightarrow Y$ in $\Oinf^\otimes$, the associated functor $f_!:\Cinf^\otimes_X\rightarrow \Cinf^\otimes_Y$ preserves small colimits.
\end{defi}

\begin{ex}
When $\Oinf^\o$ is the commutative $\infty$-operad (see \cite[2.1.1.18]{lurie1}), then a presentably $\Oinf$-monoidal $\infty$-category is called \emph{presentably symmetric monoidal}. Notice that a symmetric monoidal $\infty$-category $\Cinf$ is {presentably symmetric monoidal} if and only if $\Cinf$ is presentable and the tensor product $\otimes: \Cinf\times \Cinf\rightarrow \Cinf$ preserves small colimits in each variable.
\end{ex}

The following dualizes the result on algebras in \cite[3.2.3.5]{lurie1} and generalizes the result for cocommutative coalgebras in \cite[3.1.4]{lurie2}.

\begin{prop}\label{prop: presentability of inf coalg}
Let $\Oinf$ be an essentially small $\infty$-operad.
Let $\Cinf$ be a presentably $\Oinf$-monoidal $\infty$-category.
Then $\coalginf_\Oinf(\Cinf)$ is a presentably $\Oinf$-monoidal $\infty$-category.
\end{prop}

\begin{proof}
By Propositions \ref{prop: coalgebra are sym mon} and \ref{prop: colimits of coalgebra}, we only need to check that $\coalginf_\Oinf(\Cinf)$ is presentable.
Denote the coCartesian fibration $p:\Cinf^\o\r \Oinf^\o$ that defines the $\Oinf$-monoidal structure of $\Cinf$. 
We apply \cite[5.4.7.11, 5.4.7.14]{htt} to the subcategory $\prl$ of $\widehat{\mathcal{Cat}}_\infty$, to the Cartesian fibration $p^\vee: (\Cinf^\o)^\vee\rightarrow (\Oinf^\o)\op$ described in Remark \ref{rem: True definition of coalg in inf}, and the set of inert morphisms in $(\Oinf^\o)\op$.  The subcategory $\prl$ respects the conditions \textit{(a)}, \textit{(b)} and \textit{(c)} of \cite[5.4.7.11]{htt} by \cite[5.4.3.13, 5.5.3.6, 5.1.2.4]{htt}.

Therefore we only need to check that the fibers of $p^\vee$ over any object of $(\Oinf^\o)\op$ are presentable, and that the associated functors between the fibers, induced by the Cartesian structure of $p^\vee$, are accessible.

For any object $X$ in $\Oinf^\o$, the fiber of $p^\vee$ over $X$ is equivalent to the fiber $\Cinf^\otimes_X$ of $p$ over $X$.
By Definition \ref{defi: presentably O-monoidal}, the fibers over $p$ are presentable and the associated functors $\Cinf_X\rightarrow \Cinf_{Y}$ are accessible maps for any morphism $X\rightarrow Y$ in $\Oinf^\otimes$. Thus the induced maps $((\Cinf^\otimes)^\vee)_{Y}\rightarrow ((\Cinf^\otimes)^\vee)_X$ are also accessible, as $\Cinf^\otimes$ and $(\Cinf^\otimes)^\vee$ have the same underyling $\infty$-category $\Cinf$ by \cite[1.3]{dualcocart}. 
\end{proof}

\begin{cor}\label{cor: presentably imply presentable}
Let $\Oinf$ be an essentially small $\infty$-operad.
If $\Cinf$ is a presentably symmetric monoidal $\infty$-category, then $\coalginf_\Oinf(\Cinf)$ is a presentably symmetric monoidal $\infty$-category.
\end{cor}

\begin{rem}\label{rem: compactly generated}
In general, if $\Cinf$ is compactly generated (\cite[5.5.7.1]{htt}), there is no guarantee that $\coalginf_\Oinf(\Cinf)$ is also compactly generated. However, the \emph{fundamental theorem of coalgebras} (see \cite[II.2.2.1]{sweedler} or \cite[1.6]{GG}) states that if $\Cinf$ is (the nerve of) vector spaces, or chain complexes over a field, then $\coalginf_{\mathbb{A}_\infty}(\Cinf)$ is compactly generated and the forgetful functor $U:\coalginf_{\mathbb{A}_\infty}(\Cinf)\rightarrow \Cinf$ preserves and reflects compact objects.
From \cite[4.2]{APcoalgebra}, if $\kappa$ is an uncountable regular cardinal, we conjecture that the fundamental theorem of coalgebra can be expended in the following sense. If $\Cinf$ is $\kappa$-compactly generated then $\coalginf_\Oinf(\Cinf)$ is $\kappa$-compactly generated and the forgetful functor preserves and reflects $\kappa$-compact objects. 
\end{rem}

The forgetful functor $U:\coalginf_\Oinf(\Cinf) \rightarrow \Cinf$ admits a right adjoint functor $T^\vee:\Cinf\rightarrow \coalginf_\Oinf(\Cinf)$ called the \emph{cofree $\Oinf$-coalgebra functor}.

\begin{cor}\label{cor: cofree}
Let $\Oinf$ be an essentially small $\infty$-operad. Let $\Cinf$ be a presentably $\Oinf$-monoidal $\infty$-category. Then there is forgetful-cofree adjunction:
\[
\begin{tikzcd}[column sep= huge]
U:\coalginf_\Oinf(\Cinf)\ar[shift left=2]{r}[swap]{\perp} &\Cinf\ar[shift left=2]{l} : T^\vee.
\end{tikzcd}
\]
\end{cor}

\begin{proof}
Apply Propositions \ref{prop: colimits of coalgebra} and \ref{prop: presentability of inf coalg} and the adjoint functor theorem \cite[5.5.2.9]{htt}.
\end{proof}

In some cases, the $\infty$-category $\coalginf_\Oinf(\Cinf)$ is not mysterious. We recall the following result from Lurie. Let $\Cinf$ be a symmetric monoidal $\infty$-category, and denote by $\Cinf_\mathsf{fd}$ the full subcategory spanned by the dualizable objects, see \cite[4.6.1]{lurie1}. 
It inherits a symmetric monoidal structure. For each dualizable object $X$, we denote $X^\vee$ its dual and this defines a contravariant endofunctor on $\Cinf_\mathsf{fd}$.

\begin{prop}[{\cite[3.2.4]{lurie2}}]\label{prop: finite coalg are alg}
Let $\Cinf$ be a symmetric monoidal $\infty$-category. Then taking dual objects assigns an equivalence of symmetric monoidal $\infty$-categories $(\Cinf_\mathsf{fd})\op \stackrel{\simeq}\longrightarrow \Cinf_\mathsf{fd}$. In particular, for any $\infty$-operad $\Oinf$, we obtain an equivalence $\coalginf_\Oinf(\Cinf_\mathsf{fd})\op\simeq \alginf_\Oinf(\Cinf_\mathsf{fd})$ of symmetric monoidal $\infty$-categories.
\end{prop}

The anti-equivalence above has been generalized to a wider class in \cite[3.31]{dualitySW}.

\section{The Universal Measuring Coalgebra}

Classically, in any presentable symmetric monoidal closed ordinary category, the category of monoids is enriched, tensored and cotensored in the symmetric monoidal category of comonoids. This was proven in \cite[5.2]{hopfmeasure} and \cite[2.18]{cocat}. See also the example of the differential graded case in \cite{anel-joyal}. We  show here in Theorem \ref{thm: alg enriched over coalg} an equivalent statement in $\infty$-categories.

An $\infty$-category shall be defined to be \emph{enriched} over a symmetric monoidal $\infty$-category in the sense of \cite{gepner-haugseng}. Alternatively, the reader can use the definition in \cite[3.1.2, 7.1.1(2)]{hinichenr}. By \cite[3.4.4]{hinichenr}  they are equivalent. The author in \cite{hinichenr} uses the term \emph{precategory} instead of enriched category for various reasons that are explained in \cite[4.8]{hinichenr}, but the reader can safely ignore those technicalities and think of them as enriched $\infty$-categories for all the results we quote in this paper.

An $\infty$-category is \emph{tensored} or \emph{cotensored} over a monoidal $\infty$-category in the sense of \cite[4.2.1.19]{lurie1} or \cite[4.2.1.28]{lurie1} respectively.
Our desired enrichment in Theorem \ref{thm: alg enriched over coalg} will also be enriched in the sense of \cite[4.2.1.28]{lurie1}. It is shown in \cite{comp1} that the definitions of enrichment of Lurie and Gepner-Haugseng are equivalent. 

Throughout this section, let $\Cinf$ be a presentably symmetric monoidal $\infty$-category. 
It is in particular closed, and thus the strong symmetric monoidal functor: \[\otimes: \Cinf\times \Cinf \longrightarrow \Cinf,\] induces a lax symmetric monoidal functor $[-,-]: \Cinf\op \times \Cinf \rightarrow \Cinf$, see \cite[I.3]{runelax}, characterized by the universal mapping property
$\Cinf(X\o Y, Z) \simeq \Cinf(X, [Y, Z])$,
for all $X$, $Y$, and $Z$ in $\Cinf$. In other words, the functor $-\o Y : \Cinf\rightarrow \Cinf$ is a left adjoint to $[Y,-]:\Cinf\rightarrow \Cinf$.

\subsection{The Sweedler cotensor}
Let $\Oinf$ be an essentially small $\infty$-operad. From the lax symmetric monoidal structure of $[-,-]: \Cinf\op \times \Cinf \rightarrow \Cinf$, we obtain a functor:
 \[
 [-,-]:\alginf_\Oinf(\Cinf\op) \times \alginf_\Oinf(\Cinf) \longrightarrow \alginf_\Oinf(\Cinf).\]
By definition of $\Oinf$-coalgebras, we identify $\alginf_\Oinf(\Cinf\op)$ simply as $\coalginf_\Oinf(\Cinf)\op$, and thus obtain the following definition.

\begin{defi}
Let $\Cinf$ and $\Oinf$ be as above. We call the induced functor: \[
[-,-]:\coalginf_\Oinf(\Cinf)\op \times \alginf_\Oinf(\Cinf) \longrightarrow \alginf_\Oinf(\Cinf),
\]
the \emph{Sweedler cotensor}. In the literature, it is sometimes called the \emph{convolution algebra} or the \emph{convolution product}, see \cite[4.0]{sweedler} and \cite{anel-joyal}.
\end{defi}

\begin{rem}
The term convolution product stems from the algebra structure that generalizes the usual convolution product in representation theory. See \cite[2.12.3]{AGK}. It also generalizes the classical convolutions of real functions of compact support, see \cite[2.14.4]{AGK}.
\end{rem}

\begin{ex}
The Sweedler cotensor in the case where $\Oinf=\mathbb{E}_\infty$ and $\Cinf$ is the $\infty$-category of $R$-modules in a symmetric monoidal $\infty$-category, where $R$ is an $\mathbb{E}_\infty$-algebra, was presented in \cite[Section 1.3.1]{lurieelli2}. 
See also \cite[6.6]{nikolaus2016stable}.
\end{ex}

\begin{ex}\label{ex: linear dual of coalgebra}
Let $\I$ be the unit of the symmetric monoidal structure of $\Cinf$. Let $C$ be any $\Oinf$-coalgebra. The Sweedler cotensor $[C, \I]$ is the \emph{linear dual} $C^*$. Therefore the linear dual of an $\Oinf$-coalgebra is always an $\Oinf$-algebra. In particular the linear dual functor $(-)^*:\Cinf\op\rightarrow \Cinf$ lifts to the Sweedler cotensor $(-)^*=[-, \I]:\coalginf_\Oinf(\Cinf)\op\rightarrow \alginf_\Oinf(\Cinf)$.  Here we recover the classical result that the dual of a coalgebra is always an algebra, see \cite[1.1.1]{sweedler}. More precisely, if $C$ is a coalgebra with comultiplication $\Delta:C\rightarrow C\otimes C$ and counit $\varepsilon:C\rightarrow \I$, then $C^*$ is a coalgebra with multiplication:
\[
\begin{tikzcd}
C^* \otimes C^* \ar{r}  &(C\otimes C) ^* \ar{r}{\Delta^*} & C^*,
\end{tikzcd}
\]
where the unlabeled map is given by the lax monoidal structure of the linear dual. The unit is given by using the equivalence $\I^*\simeq \I$.
\end{ex}

\begin{rem}\label{rem: linear dual is the dual}
In a presentably symmetric monoidal $\infty$-category $\Cinf$, given an object $X$ that is dualizable (see \cite[4.6.1]{lurie1}), the dual of $X$ is given precisely by its linear dual $X^*$ (see \cite[3.10]{dualitySW}). Thus, the above defined functor $(-)^*:\coalginf_\Oinf(\Cinf)\op\rightarrow \alginf_\Oinf(\Cinf)$ coincides with the equivalence of Proposition \ref{prop: finite coalg are alg} $(-)^\vee:\coalginf_\Oinf(\Cinf_\mathsf{fd})\op\stackrel{\simeq}\longrightarrow \alginf_\Oinf(\Cinf_\mathsf{fd})$, when we restrict $(-)^*$ to the subcategory $\coalginf_\Oinf(\Cinf_\mathsf{fd})\op$.
\end{rem}

\subsection{The Sweedler tensor}
Since $[-,-]:\Cinf\op \times \Cinf \rightarrow \Cinf$ is a continuous functor in both variables, and limits in $\alginf_\Oinf(\Cinf)$ are computed in $\Cinf$, we get that the Sweedler cotensor is a continuous functor in both variables. 
Fix $C$ an $\Oinf$-coalgebra in $\Cinf$. 
Then the continuous functor: \[
[C, -]: \alginf_\Oinf(\Cinf)\rightarrow \alginf_\Oinf(\Cinf),
\]
is accessible (as filtered colimits in $\alginf_\Oinf(\Cinf)$ are computed in $\Cinf$) and is between presentable $\infty$-categories. Therefore, by the adjoint functor theorem \cite[5.5.2.9]{htt}, the functor $[C,-]$ admits a left adjoint denoted $C\sweedler -: \alginf_\Oinf(\Cinf)\rightarrow \alginf_\Oinf(\Cinf)$.

\begin{defi}
Let $\Cinf$ and $\Oinf$ be as above. We call the induced functor:
\[
-\sweedler - : \coalginf_\Oinf(\Cinf)\times \alginf_\Oinf(\Cinf)\rightarrow \alginf_\Oinf(\Cinf),
\]
the \emph{Sweedler tensor.} Previously, it was called the \emph{Sweedler product} in \cite{anel-joyal} and later in \cite{cocat}.
For $C$ a fixed $\Oinf$-coalgebra, the functor $C\sweedler -$ is left adjoint to $[C,-]$ and we have in particular the equivalence of spaces:
\[
\alginf_\Oinf(C \sweedler A, B) \simeq \alginf_\Oinf(A, [C, B] ),
\]
for any $\Oinf$-algebras $A$ and $B$.
\end{defi}

\begin{ex}
In \cite[3.4.1]{anel-joyal}, an explicit formula of the Sweedler tensor was given in the discrete differential graded case.
\end{ex}

\subsection{The Sweedler hom}
Let now $A$ be an $\Oinf$-algebra in $\Cinf$. The continuous functor: 
\[
[-,A]:\left(\coalginf_\Oinf(\Cinf)\right)\op \rightarrow \alginf_\Oinf(\Cinf),
\]
induces a cocontinuous functor on its opposites: \[
[-,A]\op: \coalginf_\Oinf(\Cinf)\rightarrow (\alginf_\Oinf(\Cinf))\op.
\] 
The cocontinuous functor is from a presentable $\infty$-category to an essentially locally small $\infty$-category: as the opposite of an essentially locally small $\infty$-category is also essentially locally small, and presentable $\infty$-categories are always essentially locally small. Thus, by the adjoint functor theorem \cite[5.5.2.9, 5.5.2.10]{htt}, the functor $[-,A]\op$ admits a right adjoint $\{-,A\}: \alginf_\Oinf(\Cinf)\op \rightarrow \coalginf_\Oinf(\Cinf)$.

\begin{defi}
Let $\Cinf$ and $\Oinf$ be as above. We call the induced functor:
\[
\{-,-\}: \alginf_\Oinf(\Cinf)\op \times \alginf_\Oinf(\Cinf) \rightarrow \coalginf_\Oinf(\Cinf),
\]
the \emph{Sweedler hom}. For $A$ and $B$ any $\Oinf$-algebra in $\Cinf$, the $\Oinf$-coalgebra $\{A, B\}$ is called the \emph{universal measuring coalgebra in $\Cinf$ of $A$ and $B$.}  See \cite[7.0]{sweedler} for the discrete case in vector spaces.
In particular, if we fix $A$, we obtain that $\{-,A\}$ is the right adjoint of $[-, A]\op$ and we have the equivalence of spaces:
\[
\coalginf_\Oinf(\Cinf)(C, \{A, B\}) \simeq \alginf_\Oinf(\Cinf) (B, [C, A]),
\]
for any $\Oinf$-coalgebra $C$.
\end{defi}

\begin{ex}\label{ex: unital measuring coalgebra}
Let $\I$ be the unit of the symmetric monoidal structure of $\Cinf$. Then, for any $\Oinf$-algebra $A$ in $\Cinf$, define $A^\circ$ to be the measuring coalgebra $\{ A, \I \}$. It is called the \emph{Sweedler dual} or \emph{finite dual} of the $\Oinf$-algebra $A$ in $\Cinf$.
In particular, we obtain a functor $(-)^\circ=\{-, \I\}\op: \alginf_\Oinf(\Cinf)\rightarrow \coalginf_\Oinf(\Cinf)\op$, which is the left adjoint of the linear dual functor $(-)^*:\coalginf_\Oinf(\Cinf)\op \rightarrow \alginf_\Oinf(\Cinf)$ defined in Example \ref{ex: linear dual of coalgebra}. 
In particular, we have the equivalence of spaces:
\[\alginf_\Oinf(\Cinf) (A, C^*) \simeq \coalginf_\Oinf(\Cinf)( C, A^\circ),\]
for any $\Oinf$-coalgebra $C$ and any $\Oinf$-algebra $A$.
This was proven in the discrete classical case of vector spaces in \cite[6.0.5]{sweedler}.
By Remark \ref{rem: linear dual is the dual}, when the $\Oinf$-algebra $A$ is dualizable in $\Cinf$, then $A^\circ\simeq A^*$ as an object in $\Cinf$.
\end{ex}

Recall we have defined the cofree $\Oinf$-coalgebra functor $T^\vee:\Cinf\rightarrow \coalginf_\Oinf(\Cinf)$ in Corollary \ref{cor: cofree}. We show that the Sweedler dual defined above of the free $\Oinf$-algebra functor $T:\Cinf\rightarrow \alginf_\Oinf(\Cinf)$ provides an explicit description of $T^\vee$. 

\begin{prop}\label{prop: cofree on double dual}
Let $\Cinf$ be a presentably symmetric monoidal $\infty$-category. Let $\Oinf$ be an essentially small $\infty$-operad. Let $X$ be an object in $\Cinf$. Then the cofree $\Oinf$-coalgebra on the double linear dual $X^{**}=(X^*)^*$ is given by:
\[
T^\vee(X^{**})\simeq\Big( T(X^*)\Big)^\circ.
\]
\end{prop}

\begin{proof}
Let $C$ be an $\Oinf$-coalgebra in $\Cinf$. By Example \ref{ex: linear dual of coalgebra} we have $U(C^*)\simeq U(C)^*$ where $U$ represents the forgetful functor on algebras or coalgebras. Then, we get the following equivalences:
\begin{eqnarray*}
\coalginf_\Oinf(\Cinf)(C, (T(X^*)^\circ)) & \simeq &\alginf_\Oinf(\Cinf)(T(X^*), C^*)\\
& \simeq & \Cinf(X^*, U(C)^*) \\
& \simeq & \Cinf(X^*\otimes U(C), \bI)\\
& \simeq & \Cinf(U(C), X^{**}).
\end{eqnarray*}
Thus we obtain the desired equivalence $T^\vee(X^{**})\simeq (T(X^*))^\circ$ by uniqueness of the right adjoint functor.
\end{proof}

The fundamental theorem of coalgebras, as seen in Remark \ref{rem: compactly generated}, provides an explicit formula for the cofree $\Oinf$-coalgebra that generalizes the approach of \cite[6.4.1]{sweedler} and \cite[1.10]{GG}.

\begin{cor}\label{cor: explicit cofree coalgebra}
Let $\Cinf$ be a presentably symmetric monoidal $\infty$-category. Let $\Oinf$ be an essentially small $\infty$-operad. Suppose $\Cinf$ is compactly generated such that $\coalginf_\Oinf(\Cinf)$ is also compactly generated. Suppose furthermore that every compact object $Y$ in $\Cinf$ is naturally equivalent to its double linear dual $Y\simeq Y^{**}$. Then, for any object $X$ in $\Cinf$, its cofree $\Oinf$-coalgebra is given by:
\[
T^\vee(X) \simeq \mathsf{colim}_i\, \left({\left(T(X_i^*)\right)}^\circ\right),
\]
where $X\simeq\colim_i X_i$ is a filtered colimit of compact objects $X_i$ in $\Cinf$.
\end{cor}

\begin{proof}
Let $C$ be an $\Oinf$-coalgebra in $\Cinf$.
Since $\coalginf_\Oinf(\Cinf)$ is compactly generated, then $C$ is the filtered colimit of compact coalgebras $C_j$:
\[
C\simeq \mathsf{colim}_j \, C_j.
\]
By compactness, we obtain the equivalence:
\[
\coalginf_\Oinf(\Cinf)\left(C, \mathsf{colim}_i \left({\left(T(X_i^*)\right)}^\circ\right) \right)\simeq \mathsf{lim}_j \, \mathsf{colim}_i\, \coalginf_\Oinf(\Cinf)(C_j, T(X_i^*)^\circ).
\]
By Proposition \ref{prop: cofree on double dual}, we obtain an equivalence:
\[
\coalginf_\Oinf(\Cinf)(C_j, T(X_i^*)^\circ) \simeq  \Cinf(U(C_j), X_i^{**}).
\]
By hypothesis, the natural equivalence $X\simeq X_i^{**}$ provides the equivalence:
\[
\Cinf(U(C_j), X_i^{**})\simeq \Cinf(U(C_j), X_i) 
\]
Since $U:\coalginf_\Oinf(\Cinf)\rightarrow \Cinf$ preserves colimits and compact objects (by Proposition \ref{prop: colimits of coalgebra} and \cite[5.5.7.2]{htt}), we obtain that $U(C)$ is the filtered colimit of compact objects $U(C_j)$ in $\Cinf$. Therefore:
\[
\mathsf{lim}_j \, \mathsf{colim}_i\, \Cinf(U(C_j), X_i) \simeq \Cinf (U(C), X).
\]
Thus we have shown $\coalginf_\Oinf(\Cinf)\left(C, \mathsf{colim}_i \left({\left(T(X_i^*)\right)}^\circ\right) \right) \simeq \Cinf (U(C), X)$. By unique\-ness of the right adjoint, we obtain the desired equivalences.
\end{proof}

We shall explain where the term \emph{universal measuring} is coming from. Recall that the internal hom property of $\Cinf$ implies that, for any $X$, $Y$ and $Z$ objects in $\Cinf$, there is an equivalence of spaces: $\Cinf(X\otimes Y, Z)\simeq \Cinf(Y, [X,Z])$.
The Sweedler cotensor gives conditions for an $\Oinf$-algebra structure on $[X, Z]$. The following is a generalization of \cite[7.0.1]{sweedler} and \cite[3.3.1]{anel-joyal}.

\begin{defi}
Let $\Cinf$ and $\Oinf$ be as above. Let $C$ be an $\Oinf$-coalgebra in $\Cinf$, and $A$ and $B$ be $\Oinf$-algebras in $\Cinf$. Let $\psi: C\otimes A\rightarrow B$ be a map in $\Cinf$. We say that \emph{$(C, \psi)$ measures $A$ to $B$} (or \emph{$(C, \psi)$ is a measuring of $A$ to $B$}) if the adjoint map $A\rightarrow [C,B]$ is a map of $\Oinf$-algebras in $\Cinf$.
\end{defi}

We give examples generalized from \cite{anel-joyal}.

\begin{ex}[{\cite[3.3.3]{anel-joyal}}]
If $\I$ is the unit of the symmetric monoidal structure of $\Cinf$, then a map $\I\otimes A\rightarrow B$ in $\Cinf$  is a measuring of $A$ to $B$ if and only if it is a map in $\alginf_\Oinf(\Cinf)$.
\end{ex}

\begin{ex}[{\cite[3.3.4]{anel-joyal}}]
The adjoint of the identity map on $[C,A]$ is a map $C\otimes [C, A] \rightarrow A$ and is always a measuring. In particular, the evaluation $C\otimes C^*\rightarrow \I$ is always a measuring of $C^*$ to $\I$. 
Similarly $A^\circ \otimes A \rightarrow \I$ is a measuring of $A$ to $\I$.
It is claimed to be the origin of the term \emph{measure} in \cite[2.12.10]{AGK}.
\end{ex}

By definition of the Sweedler hom, as we have: \[\coalginf_\Oinf(\Cinf)(C, \{A, B\}) \simeq \alginf_\Oinf(\Cinf) (B, [C, A]),\]
we see that the $\Oinf$-coalgebra $\{A, B\}$, together with the natural map $\{A, B\} \otimes A \rightarrow B$ (adjoint of the identity over $\{A, B\}$), is indeed the universal measuring algebra of $A$ to $B$, in the following sense. Given any other measuring $(C, \psi)$ of $A$ to $B$, there exists a unique (up to contractible choice) map $C\rightarrow \{A, B\}$ of $\Oinf$-coalgebras in $\Cinf$ such that the following diagram commutes in $\Cinf$:
\[
\begin{tikzcd}
C\otimes A \ar{dr}{\psi} \ar[dashed]{d}\\
\{A, B\} \otimes A \ar{r} & B.
\end{tikzcd}
\]
\begin{rem}
Following \cite[3.3.6]{anel-joyal}, we see that, given maps $A'\rightarrow A$ and $B\rightarrow B'$ in $\alginf_\Oinf(\Cinf)$, a map $C'\rightarrow C$ in $\coalginf_\Oinf(\Cinf)$, together with a map $A\rightarrow [C,B]$ in $\alginf_\Oinf(\Cinf)$, we obtain the following map in $\alginf_\Oinf(\Cinf)$:
\[
\begin{tikzcd}
A'\ar{r} & A \ar{r} & {[C,B]} \ar{r} & {[C', B']}.
\end{tikzcd}
\]
This shows that the space of measurings provides a functor: \[\coalginf_\Oinf(\Cinf)\op\times \alginf_\Oinf(\Cinf)\op\times  \alginf_\Oinf(\Cinf) \longrightarrow \mathcal{S},\]
that is representable in each variable with respect to the Sweedler hom, tensor and cotensor.
\end{rem}

\begin{rem}
We can generalize a result from \cite{measuringthh}.
Let $\Cinf$ be a presentably symmetric monoidal $\infty$-category.
Let $A$ and $B$ be  $\ei$-algebras in $\Cinf$. Recall that the topological Hochschild homology $\mathsf{THH}(A)$ is given by tensoring over the circle:
\[
\mathsf{THH}(A)\simeq A\otimes S^1.
\]
In particular, if $(C, \psi)$ is a measuring of $A$ and $B$, then from the map of $\ei$-algebras $A\rightarrow [C, B]$ we obtain a map of $\ei$-algebras:
\[
A\otimes S^1\longrightarrow [C, B]\otimes S^1 \longrightarrow [C, B\otimes S^1].
\]
Therefore $(C,\psi)$ also determines a natural measuring of $\mathsf{THH}(A)$ to $\mathsf{THH}(B)$. Therefore we obtain a map of $\ei$-coalgebras:
\[
\{A, B\}\longrightarrow \{\mathsf{THH}(A), \mathsf{THH}(B)\}.
\]
\end{rem}

\begin{rem}
From \cite[1.3.73]{anel-joyal}, the primitive elements of the measuring coalgebra $\{A, B\}$ are the derivations from $A$ to $B$. In particular, the subcoalgebra of primitive elements of the coalgebra $\{A ,A\}$ is equivalent to the tangent complex of $A$.
\end{rem}

\subsection{The enrichment in coalgebras}
Let $\Dinf^\o$ be a monoidal $\infty$-category. Its \emph{reverse}, denoted $(\Dinf^\o)^\mathsf{rev}$ or simply $\Dinf^\mathsf{rev}$, is defined in \cite[2.13.1]{hinichenr}. Essentially, $\Dinf$ and $\Dinf^\mathsf{rev}$ have the same underlying $\infty$-category but the tensor $X\o Y$ in $\Dinf^\mathsf{rev}$ corresponds precisely to $Y\o X$ in $\Dinf$. 
Left modules over $\Dinf$ correspond to right modules over $\Dinf^\mathsf{rev}$.
If $\Dinf$ is symmetric, then $\Dinf^\mathsf{rev}\simeq\Dinf$ by \cite[2.13.4]{hinichenr}. 
We shall be interested with the\emph{ reverse opposite}, denoted $\Dinf^\mathsf{rop}=(\Dinf\op)^\mathsf{rev}$, of a monoidal $\infty$-category $\Dinf$. 
The following is a generalization of the discrete ordinary case \cite[5.1]{hopfmeasure}.

\begin{lem}\label{lem: opmonoidal sweedler cotensor}
Let $\Cinf$ and $\Oinf$ be as above. Then the Sweedler cotensor endows the $\infty$-category $\alginf_\Oinf(\Cinf)$ with the structure of a right module over the reverse opposite of the (symmetric) monoidal $\infty$-category $\coalginf_\Oinf(\Cinf)$.
\end{lem}

\begin{proof}
Since the internal hom $[-,-]:\Cinf\times \Cinf\op\rightarrow \Cinf$ is a lax symmetric monoidal functor (see \cite[I.3]{runelax}), then it is a map of commutative algebras in $\widehat{\mathcal{Cat}}_\infty$, the $\infty$-category of $\infty$-categories endowed with its Cartesian monoidal structure. This shows that $\Cinf$ is a $\Cinf\op\times \Cinf$-algebra in $\widehat{\mathcal{Cat}}_\infty$, and thus in particular, $\Cinf$ is a left module over $\Cinf\op$. 
Hence  $\Cinf$ is a right module over its reverse opposite $\Cinf^\mathsf{rop}$ via its internal hom.
Therefore, by Proposition \ref{prop: coalgebra are sym mon}, the $\infty$-category $\alginf_\Oinf(\Cinf)$ is a right module over $\alginf_\Oinf(\Cinf^\mathsf{rop})$ via the Sweedler cotensor. Since $\alginf_\Oinf(\Cinf^\mathsf{rev})\simeq \alginf_\Oinf(\Cinf)^\mathsf{rev}$, then $\alginf_\Oinf(\Cinf^\mathsf{rop})\simeq \coalginf_\Oinf(\Cinf)^\mathsf{rop}$.
\end{proof}

Since $\coalginf_\Oinf(\Cinf)$ is a presentably symmetric monoidal $\infty$-category, it is enriched over itself by \cite[7.4.10]{gepner-haugseng}. 
We denote $\underline{\coalginf_\Oinf(\Cinf)}(D, E)$ the $\Oinf$-coalgebra in $\Cinf$ which classifies coalgebra maps from $D$ to $E$, characterized by the universal mapping property:
\[
\coalginf_\Oinf(\Cinf) \Big( C\o D, E\Big) \simeq\coalginf_\Oinf(\Cinf) \left( C,  \underline{\coalginf_\Oinf(\Cinf)}(D, E)\right).
\]

\begin{thm}\label{thm: alg enriched over coalg}
Let $\Cinf$ be a presentably symmetric monoidal $\infty$-category. Let $\Oinf$ be an essentially small $\infty$-operad.
The $\infty$-category of $\Oinf$-algebras $\alginf_\Oinf(\Cinf)$ is enriched over the symmetric monoidal $\infty$-category $\coalginf_\Oinf(\Cinf)$, via the Sweedler hom. Moreover it is tensored and cotensored respectively using the Sweedler tensor and Sweedler cotensor. In particular, we have an equivalence of $\Oinf$-coalgebras:
\[
\underline{\coalginf_\Oinf(\Cinf)} \Big( C, \{ A, B\} \Big) \simeq \Big\{A , [C, B] \Big\} \simeq  \Big\{ C \sweedler A, B \Big\},
\]
for any $\Oinf$-coalgebra $C$ in $\Cinf$ and any $\Oinf$-algebras $A$ and $B$ in $\Cinf$.
\end{thm}

\begin{proof}
By Lemma \ref{lem: opmonoidal sweedler cotensor}, the $\infty$-category $\alginf_\Oinf(\Cinf)\op$ is a left module over the symmetric monoidal $\infty$-category $\coalginf_\Oinf(\Cinf)$, via $[-,-]\op$ the opposite of the Sweedler cotensor, such that $[-,A]\op:\coalginf_\Oinf(\Cinf)\rightarrow \alginf_\Oinf(\Cinf)\op$ admits a right adjoint $\{-,A\}$ for all $A$ in $\alginf_\Oinf(\Cinf)$.
By \cite[7.4.9]{gepner-haugseng} (see also \cite[4.2.1.33]{lurie1} and \cite[6.3.1, 7.2.1]{hinichenr}), this shows that $\alginf_\Oinf(\Cinf)\op$ is enriched over $\coalginf_\Oinf(\Cinf)$, with tensor $[-,-]\op$. Thus, by \cite[6.2.1]{hinichenr}, we get that $\alginf_\Oinf(\Cinf)$ is enriched over $\coalginf_\Oinf(\Cinf)$, with cotensor $[-,-]$.\qedhere
\end{proof}


\begin{rem}
The previous theorem shows that we can enrich the equivalence in Example \ref{ex: unital measuring coalgebra} to an equivalence of $\Oinf$-coalgebras in $\Cinf$:
\[
\underline{\coalginf_\Oinf(\Cinf)} \Big( C,A^\circ \Big) \simeq \Big\{A , C^* \Big\} \simeq  \Big(C \sweedler A \Big)^\circ,
\]
for any $\Oinf$-coalgebra $C$ and any $\Oinf$-algebra $A$.
\end{rem}

\begin{cor}
Let $\Cinf$ be a presentably symmetric monoidal $\infty$-category. Let $\Oinf$ be an essentially small $\infty$-category. Let $A$ be an $\Oinf$-algebra in $\Cinf$. Let $C$ be an $\Oinf$-coalgebra in $\Cinf$. Then there are adjunctions of enriched $\infty$-categories over $\coalginf_\Oinf(\Cinf)$:
\begin{equation}\label{eq: I}
\begin{tikzcd}[column sep= huge]
C\sweedler -: \alginf_\Oinf(\Cinf)\ar[shift left=2]{r}[swap]{\perp} &\alginf_\Oinf(\Cinf)\ar[shift left=2]{l} : [C, -].
\end{tikzcd}
\end{equation}
\begin{equation}\label{eq: II}
\begin{tikzcd}[column sep= huge]
[-, A]\op: \coalginf_\Oinf(\Cinf)\ar[shift left=2]{r}[swap]{\perp} &\alginf_\Oinf(\Cinf)\op\ar[shift left=2]{l} : \{-, A\}.
\end{tikzcd}
\end{equation}
\begin{equation}\label{eq: III}
\begin{tikzcd}[column sep= huge]
-\sweedler A: \coalginf_\Oinf(\Cinf)\ar[shift left=2]{r}[swap]{\perp} &\alginf_\Oinf(\Cinf)\ar[shift left=2]{l} : \{A, -\}.
\end{tikzcd}
\end{equation}
\end{cor}

In \cite[5.1.2, 5.1.4]{anel-joyal}, the adjunction (\ref{eq: I}) generalizes adjunctions from Weil restrictions and on the de Rham algebra. 
The second adjunction (\ref{eq: II}) generalizes the anti-equivalence between finite dimensional algebras and finite dimensional coalgebras of Proposition \ref{prop: finite coalg are alg}, see also Remark \ref{rem: linear dual is the dual} and Example \ref{ex: unital measuring coalgebra} above. Finally, the adjunction (\ref{eq: III}) generalizes the algebraic bar-cobar adjunction as seen in \cite[5.3.14]{anel-joyal}.

\renewcommand{\bibname}{References}
\bibliographystyle{amsalpha}
\bibliography{biblio}

\providecommand{\bysame}{\leavevmode\hbox to3em{\hrulefill}\thinspace}
\providecommand{\MR}{\relax\ifhmode\unskip\space\fi MR }
\providecommand{\MRhref}[2]{%
  \href{http://www.ams.org/mathscinet-getitem?mr=#1}{#2}
}
\providecommand{\href}[2]{#2}
\begin{thebibliography}{HKRS17}

\bibitem[AJ13]{anel-joyal}
Matthieu {Anel} and Andr{\'e} {Joyal}, \emph{{Sweedler Theory for (co)algebras
  and the bar-cobar constructions}}, 2013, arXiv:1309.6952.

\bibitem[AP04]{APcoalgebra}
J.~Ad\'{a}mek and H.-E. Porst, \emph{On tree coalgebras and coalgebra
  presentations}, Theoret. Comput. Sci. \textbf{311} (2004), no.~1-3, 257--283.
  \MR{2030299}

\bibitem[BGN18]{dualcocart}
Clark Barwick, Saul Glasman, and Denis Nardin, \emph{Dualizing cartesian and
  cocartesian fibrations}, Theory Appl. Categ. \textbf{33} (2018), Paper No. 4,
  67--94. \MR{3746613}

\bibitem[BK19]{measuringthh}
Abhishek {Banerjee} and Surjeet {Kour}, \emph{{On measurings of algebras over
  operads and homology theories}}, arXiv e-prints (2019), arXiv:1909.13835.

\bibitem[BP20]{dualitySW}
{\"O}zg{\"u}r~Haldun {Bay{\i}nd{\i}r} and Maximilien {P{\'e}roux},
  \emph{{S}panier-{W}hitehead duality for topological co{H}ochschild homology},
  arXiv e-prints (2020), arXiv:2012.03966.

\bibitem[GG99]{GG}
E.~{Getlzer} and P.~{Goerss}, \emph{A model category structure for differential
  graded coalgebras}, 1999, Unpublished.

\bibitem[GH15]{gepner-haugseng}
David Gepner and Rune Haugseng, \emph{Enriched {$\infty$}-categories via
  non-symmetric {$\infty$}-operads}, Adv. Math. \textbf{279} (2015), 575--716.
  \MR{3345192}

\bibitem[{Gro}20]{groth}
Moritz {Groth}, \emph{{A short course on $\infty$-categories}}, Handbook of
  homotopy theory (Haynes Miller, ed.), CRC Press, Boca Raton, FL, 2020,
  pp.~549--618.

\bibitem[{Hau}20]{runelax}
Rune {Haugseng}, \emph{{A fibrational mate correspondence for
  $\infty$-categories}}, arXiv e-prints (2020), arXiv:2011.08808.

\bibitem[{Hei}20]{comp1}
Hadrian {Heine}, \emph{{An equivalence between enriched $\infty$-categories and
  $\infty$-categories with weak action}}, arXiv e-prints (2020),
  arXiv:2009.02428.

\bibitem[HGK10]{AGK}
Michiel Hazewinkel, Nadiya Gubareni, and V.~V. Kirichenko, \emph{Algebras,
  rings and modules}, Mathematical Surveys and Monographs, vol. 168, American
  Mathematical Society, Providence, RI, 2010, Lie algebras and Hopf algebras.
  \MR{2724822}

\bibitem[{Hin}18]{hinichenr}
V.~{Hinich}, \emph{{Yoneda lemma for enriched infinity categories}}, arXiv
  e-prints (2018), arXiv:1309.6952, to appear in {A}dvances in {M}ath.

\bibitem[HKRS17]{left2}
Kathryn Hess, Magdalena K\c{e}dziorek, Emily Riehl, and Brooke Shipley, \emph{A
  necessary and sufficient condition for induced model structures}, J. Topol.
  \textbf{10} (2017), no.~2, 324--369. \MR{3653314}

\bibitem[HLV17]{hopfmeasure}
Martin {Hyland}, Ignacio {L\'{o}pez Franco}, and Christina {Vasilakopoulou},
  \emph{Hopf measuring comonoids and enrichment}, Proc. Lond. Math. Soc. (3)
  \textbf{115} (2017), no.~5, 1118--1148. \MR{3733560}

\bibitem[Hov99]{hovey}
Mark Hovey, \emph{Model categories}, Mathematical Surveys and Monographs,
  vol.~63, American Mathematical Society, Providence, RI, 1999. \MR{1650134}

\bibitem[Lur09]{htt}
Jacob Lurie, \emph{Higher topos theory}, Annals of Mathematics Studies, vol.
  170, Princeton University Press, Princeton, NJ, 2009. \MR{2522659}

\bibitem[Lur17]{lurie1}
Jacob Lurie, \emph{Higher algebra},
  \url{https://www.math.ias.edu/~lurie/papers/HA.pdf}, 2017, electronic book.

\bibitem[Lur18a]{lurie2}
\bysame, \emph{Elliptic cohomology {I}},
  \url{https://www.math.ias.edu/~lurie/papers/Elliptic-I.pdf}, 2018,
  unpublished.

\bibitem[Lur18b]{lurieelli2}
\bysame, \emph{Elliptic cohomology {II}: Orientations.},
  \url{https://www.math.ias.edu/~lurie/papers/Elliptic-II.pdf}, 2018,
  unpublished.

\bibitem[{Nik}16]{nikolaus2016stable}
Thomas {Nikolaus}, \emph{{Stable $\infty$-Operads and the multiplicative
  {Y}oneda lemma}}, 2016, arXiv:1608.02901.

\bibitem[P{\'e}r20a]{coalginDK}
Maximilien P{\'e}roux, \emph{Coalgebras in the {D}wyer-{K}an localization of a
  model category}, arXiv e-prints (2020), arXiv:2006.09407.

\bibitem[P{\'e}r20b]{phd}
\bysame, \emph{Highly structured coalgebras and comodules},
  \url{https://indigo.uic.edu/articles/thesis/Highly_Structured_Coalgebras_and_Comodules/13475667/1},
  Aug 2020, {P}h{D} {T}hesis.

\bibitem[PS19]{perouxshipley}
Maximilien P\'{e}roux and Brooke Shipley, \emph{Coalgebras in symmetric
  monoidal categories of spectra}, Homology Homotopy Appl. \textbf{21} (2019),
  no.~1, 1--18. \MR{3852287}

\bibitem[Swe69]{sweedler}
Moss~E. Sweedler, \emph{Hopf algebras}, Mathematics Lecture Note Series, W. A.
  Benjamin, Inc., New York, 1969. \MR{0252485}

\bibitem[Vas19]{cocat}
Christina Vasilakopoulou, \emph{Enriched duality in double categories:
  {$\mathcal{V}$}-categories and {$\mathcal{V}$}-cocategories}, J. Pure Appl.
  Algebra \textbf{223} (2019), no.~7, 2889--2947. \MR{3912953}

\end{thebibliography}
\end{document}